\documentclass{imanum}

\usepackage{url,color}
\usepackage{lipsum}
\usepackage{epstopdf}
\usepackage{amsmath,amssymb}
\usepackage{amsfonts}
\usepackage{algorithm}
\usepackage{algorithmic}
\usepackage{subcaption}
\usepackage{graphicx}
\usepackage[breaklinks=true,pdfstartview=FitH,colorlinks=false]{hyperref}
\AtBeginDocument{
  \label{CorrectFirstPageLabel}
  \def\fpage{\pageref{CorrectFirstPageLabel}}
}

\jno{drnxxx}
\received{\today}

\ifpdf
  \DeclareGraphicsExtensions{.eps,.pdf,.png,.jpg}
\else
  \DeclareGraphicsExtensions{.eps}
\fi

\def\noprint#1{}

\newcommand{\bfone}{\mathbf{1}}

\newcommand{\R}{\mathbb{R}}

\def\Lmax{L_{\max}}
\def\Lmin{L_{\min}}
\def\Lavg{L_{\text{avg}}}
\def\xRPCD{x_{\text{RPCD}}}
\def\rhoRPCD{\rho_{\text{RPCD}}}
\def\xCCD{x_{\text{CCD}}}
\def\rhoCCD{\rho_{\text{CCD}}}

\def\rhoRCD{\rho_{\text{RCD}}}

\newcommand{\cP}{{\mathcal P}}

\newcommand{\E}{\mathbb{E}}

\newcommand{\ddd}{\delta}

\def\tto{\;{\lower 1pt \hbox{$\rightarrow$}}\kern -10pt
           \hbox{\raise 2.8pt \hbox{$\rightarrow$}}\;}

\newcommand{\trace}{\mbox{\rm trace}}

\newcommand{\TheTitle}{Random Permutations Fix a Worst Case for Cyclic Coordinate Descent}

\begin{document}
\title{\TheTitle}
\shorttitle{\TheTitle}

\author{%
{\sc Ching-pei Lee\thanks{ching-pei@cs.wisc.edu},
and
Stephen J. Wright\thanks{swright@cs.wisc.edu}\\[2pt]
{\sc Computer Sciences Department, University of Wisconsin-Madison,
Madison, WI.\thanks{This work was
      supported by NSF Awards DMS-1216318 and IIS-1447449, ONR Award
      N00014-13-1-0129, AFOSR Award FA9550-13-1-0138, and Subcontracts
  3F-30222 and 8F-30039 from Argonne National Laboratory.}}
}
\shortauthorlist{Ching-pei Lee and Stephen J. Wright}
}
\maketitle

\begin{abstract}
{Variants of the coordinate descent approach for minimizing a
  nonlinear function are distinguished in part by the order in which
  coordinates are considered for relaxation. Three common orderings
  are cyclic (CCD), in which we cycle through the components of $x$ in
  order; randomized (RCD), in which the component to update is
  selected randomly and independently at each iteration; and
  random-permutations cyclic (RPCD), which differs from CCD only in
  that a random permutation is applied to the variables at the start
  of each cycle.  Known convergence guarantees are weaker for CCD and
  RPCD than for RCD, though in most practical cases, computational
  performance is similar among all these variants. There is a certain
  type of quadratic function for which CCD is significantly slower
  than for RCD; a recent paper by \cite{SunY16a} has explored the poor
  behavior of CCD on functions of this type.  The RPCD approach
  performs well on these functions, even better than RCD in a certain
  regime. This paper explains the good behavior of RPCD with a tight
  analysis.}
{Coordinate descent; randomization; permutations.}
\end{abstract}



\section{Introduction} \label{sec:intro}

The basic (component-wise) coordinate descent framework for the smooth
unconstrained optimization problem
\begin{equation} \label{eq:f}
\min \, f(x), \quad \mbox{where $f:\R^n \to \R$ is smooth and convex,}
\end{equation}
is shown in Algorithm~\ref{alg:cd}. Here, we denote
\begin{equation} \label{eq:notation}
\nabla_i f(x) = [\nabla f(x)]_i, \quad
e_i=(0,\dotsc,0,1,0,\dotsc,0)^T, 
\end{equation}
where the single nonzero in $e_i$ appears in position $i$.  Each outer
cycle (indicated by index $\ell$) is called an ``epoch,'' with each
epoch consisting of $n$ iterations (indexed by $j$). The counter
$k=\ell n+j$ keeps track of the total number of iterations. At each
iteration, component $i(\ell,j)$ of $x$ is selected for updating; a
step is taken along the negative gradient direction in this component
only.

\begin{algorithm} 
\begin{algorithmic}
\STATE Choose $x^0 \in \R^n$;
\FOR{$\ell=0,1,2,\dotsc$}
\FOR{$j=0,1,2,\dotsc,n-1$}
\STATE Define $k=\ell  n+j$;
\STATE Choose index $i=i(\ell,j) \in \{1,2,\dotsc,n\}$;
\STATE Choose  $\alpha_k>0$;
\STATE $x^{k+1} \leftarrow x^k - \alpha_k \nabla_{i} f(x^k) e_{i}$;
\ENDFOR
\ENDFOR
\end{algorithmic}
\caption{Coordinate Descent\label{alg:cd}}
\end{algorithm}

There are several variants within this simple framework. One important
source of variation is the choice of coordinate $i=i(\ell,j)$. Three
popular choices are as follows:
\begin{itemize}
\item CCD (Cyclic CD): $i(\ell,j)=j+1$.
\item RCD (Randomized CD, also known as Stochastic CD): $i(\ell,j)$ is
  chosen uniformly at random from $\{1,2,\dotsc,n\}$ --- 
  sampling-with-replacement.
\item RPCD (Random-Permutations Cyclic CD): At the start of epoch
  $\ell$, we choose a random permutation of $\{1,2,\dotsc,n\}$,
  denoted by $\pi_{\ell+1}$. Index $i(\ell,j)$ is chosen to be the $(j+1)$th
  entry in $\pi_{\ell+1}$. This approach represents sampling without
  replacement, within each epoch.
\end{itemize}
(Other ways to choose $i(\ell,j)$ include weighted forms of RCD, in
which $i(\ell,j)$ is selected from a nonuniform distribution; and a
Gauss-Southwell form, in which $i(\ell,j)$ is the component that
maximizes $| \nabla_i f(x^k)|$.)

When $f$ is a convex quadratic function, and when $\alpha_k$ in
Algorithm~\ref{alg:cd} is chosen to minimize $f$ exactly along each
coordinate direction, these variants are simply different variants of
the Gauss-Seidel approach for solving the equivalent system of linear
equations.

The coordinate descent approach is enjoying renewed popularity because
of its usefulness in data analysis applications. Its convergence
properties have come under renewed scrutiny. We refer to  
\citep{Wri14d} for a discussion of the state of the art as of 2015, but
make a few additions and updates here, with an emphasis on results
concerning linear convergence of the function values, by which we mean
epoch-wise convergence of the form
\begin{equation} \label{eq:lin}
f(x^{(\ell+1) n}) - f^* \le \rho (f(x^{\ell n})-f^*) \quad \mbox{for
  some $\rho \in (0,1)$},
\end{equation}
where $\rho$ is typically much closer to $1$ than to $0$, and $f^*$ is
the optimal value of \eqref{eq:f}.  For randomized methods, we
consider a corresponding expression {\em in expectation:}
\begin{equation} \label{eq:lin.e}
\E \left[ f(x^{(\ell+1) n}) - f^* \right] \le \rho \E \left[ f(x^{\ell n})-f^* \right],
\end{equation}
where the expectation is taken over all random variables encountered
in the algorithm.  When \eqref{eq:lin} holds, a reduction in function
error by a factor of $\epsilon$ can be attained in approximately
$|\log \epsilon|/(1-\rho)$ epochs. We sometimes refer to
$1/(1-\rho)$ as the ``complexity'' of an algorithm for which
\eqref{eq:lin} or \eqref{eq:lin.e} holds.

\subsection{Characterizing the Objective}

We preface a discussion of linear convergence rates with some
definitions of certain constants associated with $f$. We assume for
simplicity that the domain of $f$ is the full space $\R^n$.  The
component Lipschitz constants $L_i$, $i=1,2,\dotsc,n$ satisfy
\begin{equation} \label{eq:Li}
\left| \nabla_i f(x+te_i) - \nabla_i f(x) \right| \le L_i \left|t\right|, \quad
\mbox{for all $x\in \R^n$ and $t \in \R$.}
\end{equation}
We have
\begin{equation} \label{eq:Lmax}
\Lmax := \max_{i=1,2,\dotsc,n} \, L_i,
\quad \Lmin := \min_{i=1,2,\dotsc,n} \, L_i, \quad
\Lavg := \sum_{i=1}^n L_i/n.
\end{equation}
The standard Lipschitz constant $L$ is defined so that
\begin{equation} \label{eq:L}
\left\| \nabla f(x+d) - \nabla f(x) \right\| \le L \|d \|, \quad \mbox{for all $x,d \in \R^n$.}
\end{equation}
(Here and throughout, we use $\| \cdot \|$ to denote $\| \cdot \|_2$.)
For reasonable choices of the constants in \eqref{eq:Li},
\eqref{eq:Lmax}, and \eqref{eq:L}, the following bounds are satisfied:
\begin{equation} \label{eq:LmaxL}
1 \le \frac{L}{\Lmax} \le n.
\end{equation}

The following property of 
\cite{Loj63} is useful in
proving linear convergence:
\begin{equation}
\label{eq:strongProperty}
\| \nabla f(x) \|^2 \ge 2\mu [f(x)-f^*], \quad \mbox{for some $\mu>0$.}
\end{equation}
This property holds for $f$ strongly convex (with modulus of strong
convexity $\mu$), and for the case in which $f$ grows quadratically
with distance from a non-unique minimizing set, as in the ``optimal
strong convexity'' condition of \citet[(1.2)]{LiuW14c}. It also holds
generically for convex quadratic programs, even when the Hessians are
singular.
Further, condition \eqref{eq:strongProperty} holds for the functional
form considered by 
\cite{LuoT92a,LuoT93a}, which is
\begin{equation} \label{eq:fgE}
f(x) = g(Ex), \quad \mbox{where $E \in \R^{m \times n}$ and $g: \R^m
  \to \R$ strongly convex,}
\end{equation}
without any conditions on $E$. (For a proof, see
Appendix~\ref{app:gEx}.)  In \citep{KarNS16a}, property
\eqref{eq:strongProperty} is called the ``Polyak-{\L}ojasiewicz
condition.''

In this paper, our focus is on the case of $f$ convex quadratic, that
is,
\begin{equation} \label{eq:q}
f(x) = \frac12 x^TAx, \quad \mbox{where $A$ is symmetric positive
  semidefinite.}
\end{equation}
For this function, the  values of $L_i$, $L$, $\Lmax$, and $\mu$ are as follows:
\begin{equation} \label{eq:LmaxL.A}
\mu=\lambda_{\text{min,nz}}(A), \quad
L_i = A_{ii}; \; i=1,2,\dotsc,n; \quad \Lmax = \max_{i=1,2,\dotsc,n} A_{ii};
\quad L = \|A\|_2;
\end{equation}
where $\lambda_{\text{min,nz}}(\cdot)$ denotes the minimum nonzero
eigenvalue.  For such functions, the upper bound in \eqref{eq:LmaxL}
is achieved by $A=\bfone \bfone^T$ (where $\bfone=(1,1,\dotsc,1)^T$),
for which $L_i=1$, $i=1,2,\dotsc,n$; $\Lmax=1$; and $L =n$.

We have not included a linear term in \eqref{eq:q}, but note that
there is no loss of generality in doing so. If we were to consider instead
\[
f(x) = \frac12 x^TAx - b^Tx = \frac12 (x-x^*)^TA(x-x^*) - \frac12 b^T A^{-1}b, \quad
	\mbox{where $x^*=A^{-1} b$,}
\]
(note that $x^*$ is the minimizer of this function), the main results of
Sections~\ref{sec:comp} and ~\ref{sec:rpcd} would continue to hold,
except that in several theorems the initial iterate $x^0$ would be
replaced by $x^0-x^*$, and $f(x)$ is replaced by $f(x)-f(x^*)$.

\subsection{Linear Convergence Results for CD Variants}

\cite{LuoT92a} prove linear convergence for a function
of the form \eqref{eq:fgE}, where they require $E$ to have no zero
columns. They obtain expressions for the constant $\rho$ in
\eqref{eq:lin} for two variants of CD --- a Gauss-Southwell variant
and an ``almost cyclic'' rule --- but these constants are difficult to
characterize in terms of fundamental properties of $f$. In
\citep{LuoT93a}, the same authors analyze a family of methods
(including CD) for more general functions that satisfy a local error
bound of the form $\| x - P(x) \| \le \chi \| \nabla f(x) \|$ holds
(where $P(x)$ is the projection of $x$ onto the solution set of
\eqref{eq:f} and $\chi$ is some constant). Again, their analysis is
not clear about how the constant $\rho$ of \eqref{eq:lin} depends on
the properties of $f$.




A family of linear convergence results is proved in
\citet[Theorem~3.9]{BecT12a} for the case in which $f$ is strongly
convex (immediately extendable to the case in which $f$ satisfies the
condition \eqref{eq:strongProperty}). For constant stepsizes $\alpha_k
\equiv \alpha \le 1/\Lmax$, convergence of the form \eqref{eq:lin}
holds with
\begin{equation} \label{eq:bect.1}
\rho \leq 1- \frac{\mu}{(2/\alpha) (1+n L^2 \alpha^2)}.
\end{equation}
In particular, for $\alpha=1/L$, we have $\rho \leq
1-\mu/(2L(n+1))$. The upper bound on $\rho$ is optimized by steplength
$\alpha=1/(\sqrt{n}L)$, for which $\rho \leq 1-\mu/(\sqrt{n}L)$.  For
the case in which $f$ is a convex quadratic \eqref{eq:q} and an exact
line search is performed at each iteration (that is, $\alpha_k =
1/A_{ii}$, where $i=i(\ell,j)$ is the index to be updated in iteration
$j$ of Algorithm~\ref{alg:cd}), \citet[(3.23)]{BecT12a} show that $\rho
\leq 1-\mu/(2\Lmax(1+n^2L^2/\mu^2))$ in expression
\eqref{eq:lin}. Paradoxically, as noted by \cite{SunY16a},
use of the exact steplength leads to a considerably slower rate bound
than the conservative fixed choices. The bound for this case is
improved in \citep{SunY16a} to
\begin{equation} \label{eq:suny}
	\rho \leq 1- \max \left\{
		\frac{\mu \Lmin}{n L \Lavg},
	        \frac{\mu \Lmin}{L^2 ( 2 +\log n /\pi)^2},
		\frac{\mu \Lmin}{n^2 \Lavg^2} \right\},
\end{equation}

For the random-permutations cyclic version RPCD, the convergence
theory in \citep{BecT12a} can be applied without modification to attain
the bounds given above. As we discuss below, however, the practical
performance of RPCD is sometimes much better than these bounds would
suggest.

Convergence of the sampling-with-replacement variant RCD for
strongly convex unconstrained problems was analyzed by \cite{Nes12a}.  It
follows from the convergence theory of \citet[Theorem~2]{Nes12a} that
\eqref{eq:lin.e} holds over the i.i.d.  uniformly random choices of
indices $i(\ell,j)$ with
\begin{equation} \label{eq:rcd.rate}
\rho \le \left( 1- \frac{\mu}{n\Lmax} \right)^n \approx 1-
\frac{\mu}{\Lmax}.
\end{equation}
A different convergence rate is proved in \citet[Theorem~5]{Nes12a}, namely,
\begin{equation} \label{eq:rcd.rate.R}
\E (f(x^k) - f^*) \le C \left( 1- \frac{2 \mu}{n(\Lmax+\mu)} \right)^k,
\end{equation}
for some constant $C$ depending on the initial point. This is an
R-linear expression, obtained from Q-linear convergence of the
modified function $f(x) - f(x^*) + \sum_i L_i (x_i - x^*_i)^2/2$, where $x^*$ is the
(unique) solution of \eqref{eq:f}.  It indicates a complexity of
approximately $|\log \epsilon|(\Lmax + \mu) /(2\mu)$.

An important benchmark in studying the convergence rates of coordinate
descent is the steepest-descent (SD) method, which takes a step from
$x^k$ along all coordinates simultaneously, in the direction $-\nabla
f(x^k)$. For some important types of functions, including
empirical-risk-minimization functions that arise in data analysis, the
computational cost of one steepest-descent step is comparable to the
cost of one epoch of Algorithm~\ref{alg:cd} (see \citep{Wri14d}).
Standard analysis of steepest descent shows that fixed-steplength
variants applied to functions satisfying \eqref{eq:strongProperty}
have linear convergence of the form \eqref{eq:lin} (with one iteration
of SD replacing one epoch of Algorithm~\ref{alg:cd}) with $\rho =
1-\mu/L$. This worst-case complexity is not improved qualitatively by
using exact line searches.


In comparing convergence rates between CCD and SD (on the one hand)
and RCD (on the other hand), we see that the former tend to depend on
$L$ while the latter depends on $\Lmax$. These bounds suggest that CCD
may tend to track the performance of SD, while RCD could be
significantly better if the ratio $L/\Lmax$ is large, that is, toward
the upper end of its range in \eqref{eq:LmaxL}.  The phenomenon of
large values of $L/\Lmax$ is captured well by convex quadratic
problems \eqref{eq:q} in which the Hessian $A$ has a large
contribution from $\bfone \bfone^T$. Such matrices were used in
computations by one of the authors in 2015 (see \citep{Wri15f};
reported briefly in \citep{Wri14d}). These tests showed that on such
matrices, RCD was indeed much faster than CCD (and also SD). The
performance of RPCD was as fast as that of RCD; it did not track CCD
as the obvious worst-case analysis would suggest. Later work, reported
in \citep{Wri15g}, identified the matrix
\begin{equation} \label{eq:Ainvariant}
A := \ddd I + (1-\ddd) \bfone \bfone^T, \mbox{where $\ddd \in
  (0,n/(n-1))$,}
\end{equation}
(where $\bfone=(1,1,\dotsc,1)^T$) as being the archetype of a problem
with large ratio $L/\Lmax$. This matrix has one dominant eigenvalue
$\ddd+n(1-\ddd)$ with eigenvector $\bfone$, with the other $(n-1)$
eigenvalues equal to $\ddd$. (This matrix also has $P^TAP=A$ for all
permutation matrices $P$ --- a property that greatly simplifies the
analysis of RPCD variants, as we see below.) In \citep{Wri15g}, the
RPCD variant was shown to be significantly superior to the CCD in
computational tests.  Independently, \cite{SunY16a} studied this same
matrix \eqref{eq:Ainvariant}, using analysis to explain the practical
advantage of RCD over CCD, and showing that the performance of CCD
approaches its worst-case theoretical bound. RPCD is also studied in
\citet[Proposition~3.4, Section~C.2]{SunY16a}, the results suggesting
similar behavior for RPCD and RCD on the problem \eqref{eq:q},
\eqref{eq:Ainvariant}. However, these results are based on upper
bounds on the quantity $\| \E(x^k) \|$. By Jensen's inequality, this
quantity provides a lower bound for $\E f(x^k)$ and also for
$\E \|x^k\|^2$, but not an upper bound. (The latter is the focus of
this paper.)

The matrix \eqref{eq:Ainvariant} is also studied in \citep{YosS15a},
which investigates the tightness of the worst-case theoretical
Q-linear convergence rate for RCD applied to the problem
\eqref{eq:q}, \eqref{eq:Ainvariant} proved in \citep{Nes12a}. This
paper shows a lower bound for $O(\|\E[x^k]\|)$ for RCD, but not for
the expected objective value.

\subsection{Motivation and Outline} \label{sec:outline}

Our focus in this paper is to analyze the performance of RPCD for
minimizing \eqref{eq:q} with $A$ defined in \eqref{eq:Ainvariant}.
Our interest in RPCD is motivated by computational practice. Much has
been written about randomized optimization algorithms (particularly
stochastic gradient and coordinate descent) in recent years. The
analysis usually applies to sampling-with-replacement versions, but
the implementations almost always involve a
sampling-without-replacement scheme. The reasons are clear:
Convergence analysis is much more straightforward for sampling with
replacement, while for sampling {\em without} replacement,
implementations are more efficient, involving less data movement.
Moreover, it has long been folklore in the machine learning community
that sampling-without-replacement schemes perform better in
practice. In this paper we take a step toward bringing the analysis
into line with the practice, by giving a tight analysis of the
sampling-without-replacement scheme RPCD, on a special but important
function that captures perfectly the advantages of randomized schemes
over a deterministic scheme.

In Section~\ref{sec:comp}, we derive tools for analyzing epoch-wise
convergence of CD variants on convex quadratic problems \eqref{eq:q},
focusing on the permutation-invariant matrix \eqref{eq:Ainvariant} and
recalling results for the CCD and RCD variant in this case (obtained
from \citep{SunY16a} and \citep{Nes12a}). Section~\ref{sec:rpcd}
contains our results for RPCD applied to \eqref{eq:q} with the
permutation-invariant matrix \eqref{eq:Ainvariant}, characterizing its
convergence rate in terms of a two-parameter recurrence. The
relationship of this two-parameter sequence to the expected function
value at the end of each epoch is described in
Theorem~\ref{th:RPCD.f}. Our main result, Theorem~\ref{th:rpcd.main},
gives bounds on these two parameters in terms of $\ddd$ (the parameter
that defines \eqref{eq:Ainvariant}) and epoch number. These bounds
indicate that the convergence rate of RPCD matches that of RCD, and
both are much faster than CCD on the problem defined by \eqref{eq:q}
and \eqref{eq:Ainvariant}. We also note that a slightly tighter bound
on the asymptotic behavior of the two-variable recurrence can be
obtained from the spectral radius of the $2 \times 2$ matrix governing
this recurrence, in a regime in which $\ddd$ is close to zero. We
derive an estimate of this spectral radius in \eqref{eq:rhoM}, using
results from Appendix~\ref{app:m2}.  Theorem~\ref{th:first} explores
the behavior of the randomized methods on the very first iteration,
showing that a significant decrease can be expected just on this one
iteration. (Similar results can be expected for the cyclic variant
CCD, as we remark in comments following Theorem~\ref{th:first}.)

Empirical verification of our analysis of RPCD, and computational
comparisons with CCD and RCD, are presented in
Section~\ref{sec:computations}. The theoretical results are confirmed
nicely in all cases. We conclude with some discussions in
Section~\ref{sec:conclusions}.

\section{Convergence of CD Variants on Convex Quadratics} \label{sec:comp}

We consider the application of CCD and RPCD to the convex quadratic
problem \eqref{eq:q}.  This problem has solution $x^*=0$ with optimal
objective $f^*=0$. We assume that the matrix $A$ is diagonally scaled so
that
\begin{equation} \label{eq:diag.scaled}
A_{ii}=1, \quad \mbox{for $i=1,2,\dotsc,n$.}
\end{equation}
Under this assumption, the step of Algorithm~\ref{alg:cd} with exact
line search will have the form
\begin{equation} \label{eq:q.exact}
x^{k+1} = x^k - \frac{1}{A_{ii}} (Ax^k)_{i} e_{i} = x^k -
(Ax^k)_{i} e_{i}, \quad
\mbox{with $k=\ell n+j$ and $i=i(\ell,j)$.}
\end{equation}
Some variants of CD methods applied \eqref{eq:q} can be viewed as
Gauss-Seidel methods applied to the system $Ax=0$. Cyclic CD
corresponds to standard Gauss-Seidel, whereas RCD and RPCD are
variants of randomized Gauss-Seidel.



\subsection{CCD and RPCD Convergence Rates: General $A$} 

\label{sec:genA}

Writing $A=L+D+L^T$, where $L$ is strictly lower triangular and $D$ is
the diagonal, one epoch of the CCD method can be written as follows:
\begin{equation} \label{eq:splitC}
	x^{(\ell+1)n} = -(L+D)^{-1}(L^Tx^{\ell n}) = C x^{\ell n}, \quad \mbox{where
  $C:= -(L+D)^{-1} L^T$}.
\end{equation}
By applying the formula \eqref{eq:splitC} recursively, we obtain the
following expression for the iterate generated after $\ell$ epochs of
CCD:
\begin{equation} \label{eq:xGS}
\xCCD^{\ell n} = C^{\ell} x^0, \quad
f(\xCCD^{\ell n}) = \frac12 (x^0)^T (C^T)^{\ell} A C^{\ell} x^0.
\end{equation}
The average improvement in $f$ per epoch is obtained from the formula
\begin{equation} \label{eq:GS.av}
\rhoCCD(A,x^0) := \lim_{\ell \to \infty} \left(
f(\xCCD^{\ell n}) / f(x^0) \right)^{1/\ell}.
\end{equation}
To obtain a bound on this quantity, we denote the eigenvalues of $C$
by $\gamma_i$, $i=1,2,\dotsc,n$, and recall that the spectral radius
$\rho(C)$ is $\max_{i=1,2,\dotsc,n} \, | \gamma_i|$. Since $A$ is
positive definite, we have $\rho(C)<1$
\cite[Theorem~11.2.3]{GolV12a}. We have from Gelfand's formula
\citep{Gel41a} that
\begin{equation} \label{eq:gelfand}
\rho(C) = \lim_{\ell \to \infty} \| C^\ell \|^{1/\ell}.
\end{equation}
We can obtain a bound on $\rhoCCD(A,x^0)$ in terms of $\rho(C)$ as
follows:
\begin{align}
\nonumber
\rhoCCD (A,x^0) &:= \lim_{\ell \to \infty} \, 
\left( f(\xCCD^{\ell n})/f(x^0) \right)^{1/\ell} \\
\nonumber
&= \lim_{\ell \to \infty} \, 
\left( x_0^T (C^T)^{\ell} A C^{\ell} x_0 / x_0^T Ax_0 \right)^{1/\ell}  \\
\nonumber
&= \lim_{\ell \to \infty} \, 
\left( \| A^{1/2} C^{\ell} x_0 \|_2^2 / \| A^{1/2} x_0 \|_2^2 \right)^{1/\ell} \\
\nonumber
&= \lim_{\ell \to \infty} \, 
\left( \|( A^{1/2} C^{\ell} A^{-1/2}) (A^{1/2} x_0) \|_2^2 / \| A^{1/2} x_0 \|_2^2 \right)^{1/\ell} \\
\nonumber
&\le \lim_{\ell \to \infty} \, 
\left( \| A^{1/2} C^{\ell} A^{-1/2} \|_2^2 \right) ^{1/\ell} \\
\label{eq:fGS.av}
&\le \lim_{\ell \to \infty} \, 
\mbox{cond}(A)^{1/\ell} \| C^{\ell} \|^{2/\ell} = \rho(C)^2.
\end{align}



We can describe each epoch of RPCD algebraically by using a
permutation matrix $P_{l}$ to represent the permutation $\pi_{l}$ on
epoch $l-1$. We split the matrix $P_{l}^T A P_{l}$ and define the
operator $C_{l}$ as follows:
\begin{equation} \label{eq:Cl}
P_{l}^T A P_{l} = L_{l} + D_{l} + L_{l}^T, \quad
C_{l} := -\left(L_{l}+D_{l}\right)^{-1} L_{l}^T.
\end{equation}
The iterate generated after
$\ell$ epochs of RPCD is
\begin{equation} \label{eq:xRGS}
\xRPCD^{\ell n} = P_{\ell} C_{\ell} P_{\ell}^T P_{\ell-1} C_{\ell-1}
P_{\ell-1}^T \dotsc P_1 C_1 P_1^T x^0.
\end{equation}
(Note that in epoch $l-1$, the elements of $x$ are permuted according
to the permutation matrix $P_l$, then operated on with $C_l$ then the
permutation is reversed with $P_l^T$.)  The function value after
$\ell$ epochs is
\begin{equation} \label{eq:ffg}
f\left(\xRPCD^{\ell n}\right) = \frac12 \left(x^0\right)^T \left( 
P_1 C_1^T P_1^T  \dotsc P_{\ell} C_{\ell}^T P_{\ell}^T A
P_{\ell} C_{\ell} P_{\ell}^T \dotsc P_1 C_1 P_1^T \right) x^0.
\end{equation}
If we could take the expected value of this quantity over all random
permutations $P_1,P_2, \dotsc, P_{\ell}$, 
we would have good expected-case bounds on the convergence of
RPCD. This expectation is quite difficult to manipulate in general
(though, as we see below, is it not so difficult for $A$ defined by
\eqref{eq:Ainvariant}).
When the elements of $x^0$ are distributed according to $N(0,1)$, we
have
\begin{equation} \label{eq:ffg.x0}
\E_{x^0} f(\xRPCD^{\ell n}) = \frac12 \trace \left( 
P_1 C_1^T P_1^T  \dotsc P_{\ell} C_{\ell}^T P_{\ell}^T A
P_{\ell} C_{\ell} P_{\ell}^T \dotsc P_1 C_1 P_1^T \right).
\end{equation}

\begin{figure}\centering
\includegraphics[width=0.7\linewidth]{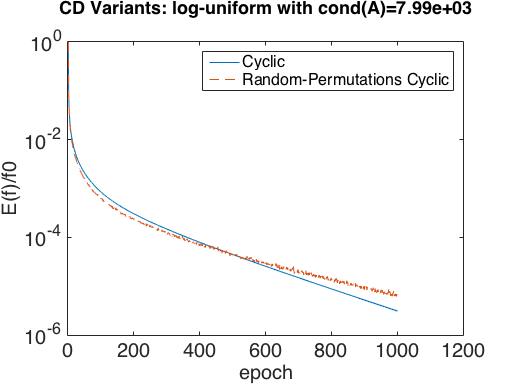}
\caption{CCD and RPCD on convex quadratic objective, for log-uniform
  eigenvalue distribution.\label{fig:lu}}
\end{figure}

Figure~\ref{fig:lu} shows typical computational results of the CCD and
RPCD variants of Algorithm~\ref{alg:cd} in the case in which the
eigenvalues of $A$ follow a log-uniform distribution, with $\kappa(A)
\approx 10^4$. The eigenvectors form an orthogonal matrix with random
orientation. Here we plot the relative {\em expected} values with
respect to $x^0$ of the $f$ on the vertical axis, that is, $\E_{x^0}
(f(x^{\ell n}))/\E_{x^0}(f(x^0))$ (see \eqref{eq:ffg.x0} for $\E_{x^0}
f(\xRPCD^{\ell n})$; similar formulas apply for $\E_{x^0}
f(\xCCD^{\ell n})$ and $\E_{x^0} f(x^0)$). This figure captures the
typical relative behavior of CCD and RPCD for ``benign'' distributions
of eigenvalues: There is little difference in performance between the
two variants.

\subsection{CD Variants Applied to Permutation-Invariant $A$} \label{sec:pinvA}

In our search for the simplest instance of a matrix $A$ for which the
superiority of randomization is observed, we arrived at the matrix
\eqref{eq:Ainvariant}.
As mentioned above, the eigenvalues of $A$ are
\[
\ddd+n(1-\ddd),\ddd,\ddd,\dotsc,\ddd, \quad \mbox{where $\ddd \in
  (0,n/(n-1))$.}
\]
The restriction in \eqref{eq:Ainvariant} ensures that $A$ has the
following properties:
\begin{itemize}
\item symmetric and positive definite;
\item unit diagonals: $A_{ii}=1$, $i=1,2,\dotsc,n$;
\item invariant under symmetric permutations of the rows and columns,
  that is, $P^TAP=A$ for any $n \times n$ permutation matrix $P$;
\item $L/\Lmax$ is close to its maximum value of $n$ when $\ddd$ is
  small, opening a wide gap between the worst-case theoretical
  behaviors of CCD and RCD.
\end{itemize}


\begin{figure}\centering
\begin{subfigure}[b]{0.45\textwidth}
\includegraphics[width=\linewidth]{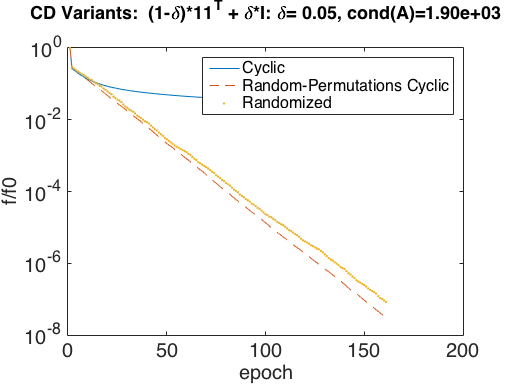}
\caption{$\ddd = .05$}
\end{subfigure}
\begin{subfigure}[b]{0.45\textwidth}
\includegraphics[width=\linewidth]{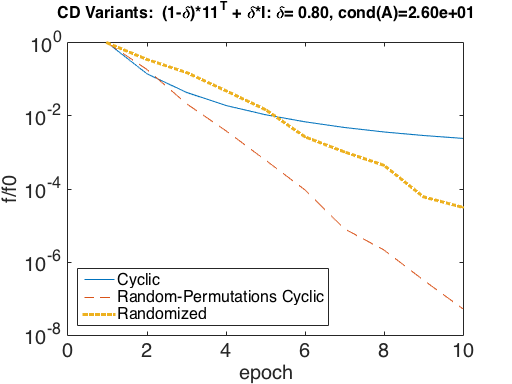}
\caption{$\ddd = .8$}
\label{subfig:deltalarge}
\end{subfigure}
\caption{CCD, RPCD, and RCD on convex quadratic objective, with $A$
	defined by \eqref{eq:Ainvariant} with
	$n=100$ and various $\ddd$.\label{fig:onebig_actual}}
\end{figure}

Figure~\ref{fig:onebig_actual} shows results for the CCD, RPCD, and
RCD variants on the matrix $A$ from \eqref{eq:Ainvariant} with $n=100$
and two different values of $\ddd$. Here, the vertical axis shows
actual function values (not expected values) relative to $f(x^0)$, for
some particular $x^0$ whose elements are drawn i.i.d from $N(0,1)$.
For both values of $\ddd$, both randomized variants are much faster
than CCD. For the larger value of $\ddd$, RPCD has a clear advantage
over RCD. Our analysis below supports these empirical observations.

We now derive expressions for the epoch iteration matrix $C$ of
Section~\ref{sec:genA} for the specific case of the
permutation-invariant matrix \eqref{eq:Ainvariant}.  This is needed
for the analysis of RPCD on this matrix.
By applying the splitting \eqref{eq:splitC} to \eqref{eq:Ainvariant},
we have
\begin{equation} \label{eq:LD}
D = I,\quad L = (1-\ddd)
E, \quad \mbox{where} \;\; E = \left[ \begin{matrix} 0 & 0 & 0 & \dotsc & 0 & 0 \\
1 & 0 & 0 & \dotsc & 0 & 0\\
1 & 1 & 0 & \dotsc & 0 & 0\\
\vdots & \vdots & \vdots & & \vdots & \vdots \\
1 & 1 & 1 & \dotsc & 1 & 0
\end{matrix} \right].
\end{equation}
Thus, defining
\begin{equation} \label{eq:Lbar}
\bar{L} := -(L+D)^{-1},
\end{equation}
we have
\begin{subequations}
\label{eq:Cinvariant}
\begin{align} 
\label{eq:Lbar2}
\bar{L}_{ij} &= \begin{cases}
-1 & \;\; \mbox{if $i=j$} \\
(1-\ddd)\ddd^{i-j-1} & \;\; \mbox{if $i>j$} \\
0 & \;\; \mbox{if $i<j$,}
\end{cases} \\
\label{eq:Cinv2}
 C & = (1-\ddd) \bar{L}E^T.
\end{align}
\end{subequations}
Writing $\bar{L}$ explicitly, we have
\[
\bar{L} = \left[
\begin{matrix}
-1 & 0 & 0 & 0 & \dotsc & 0 \\
(1-\ddd) & -1 & 0 & 0 & \dotsc & 0 \\
(1-\ddd)\ddd & (1-\ddd) & -1 & 0 & \dotsc &  0 \\
(1-\ddd) \ddd^2 & (1-\ddd) \ddd & (1-\ddd) & -1 & \dotsc & 0 \\
\vdots & \vdots & \vdots & & \ddots & \vdots \\
(1-\ddd) \ddd^{n-2} & (1-\ddd) \ddd^{n-3} & (1-\ddd) \ddd^{n-4} & \dotsc & \dotsc & -1 
\end{matrix}
\right].
\]
We have from \eqref{eq:Cinv2} and the properties of $E$ and $\bar{L}$ that
\[
C_{ij} = (1-\ddd) \sum_{\ell=1}^n \bar{L}_{i \ell} E_{j \ell} = (1-\ddd) \sum_{\ell=1}^{\min(i,j-1)} \bar{L}_{i \ell}.
\]
Thus for $i<j$ we have
\begin{align*}
C_{ij} &= (1-\ddd) \sum_{\ell=1}^i \bar{L}_{i\ell} \\
&= (1-\ddd) \left[ (1-\ddd) (\ddd^{i-2}+\ddd^{i-3}+ \dotsc+\ddd+1) - 1 \right] \\
&= (1-\ddd) \left[ (1-\ddd) \frac{1-\ddd^{i-1}}{1-\ddd} - 1 \right] \\
&= -(1-\ddd) \ddd^{i-1}.
\end{align*}
For the complementary case $i \ge j$, we have
\begin{align*}
C_{ij} &= (1-\ddd) \sum_{\ell=1}^{j-1} \bar{L}_{i \ell} \\
&= (1-\ddd) \left[ (1-\ddd) (\ddd^{i-2} + \ddd^{i-3} + \dotsc + \ddd^{i-j}) \right] \\
&= (1-\ddd)^2 \ddd^{i-j} (\ddd^{j-2}+\ddd^{j-3} + \dotsc + 1) \\
&= (1-\ddd)^2 \ddd^{i-j} \frac{1-\ddd^{j-1}}{1-\ddd} \\
&= (1-\ddd) \ddd^{i-j} (1-\ddd^{j-1}) \\
&= (1-\ddd) (\ddd^{i-j} - \ddd^{i-1}).
\end{align*}
To summarize, we have
\begin{equation} \label{eq:C.elts}
C_{ij} = \begin{cases}
-(1-\ddd) \ddd^{i-1} & \;\; \mbox{for $i<j$} \\
(1-\ddd) (\ddd^{i-j}-\ddd^{i-1}) & \;\; \mbox{for $i \ge j$.}
\end{cases}
\end{equation}

\subsection{Convergence Rates of CCD and RCD on the Permutation-Invariant $A$} \label{sec:ccdA}

Here, we examine the theoretical convergence rate of CCD on the
quadratic function with Hessian \eqref{eq:Ainvariant} by using the
results of \cite{SunY16a}.

Recalling the rate \eqref{eq:suny} from
\citet[Proposition~3.1]{SunY16a}, and substituting the following
quantities for  \eqref{eq:Ainvariant}:
\begin{equation}
	L = n(1 - \ddd) + \ddd, \quad \Lmin = 1,\quad 
        \Lavg=1, \quad \mu = \ddd,
\label{eq:consts}
\end{equation}
we find that
\[
\rhoCCD (\ddd,x^0) \le
1-\max \left\{ \frac{\ddd}{n (n(1-\ddd)+\ddd)}, \frac{\ddd}{(n(1-\ddd)+\ddd)^2 (2+\log n/\pi)^2}, \frac{\ddd}{n^2} \right\}.
\]
(We use $\rhoCCD(\ddd,x^0)$ in place of $\rhoCCD(A,x^0)$, to emphasize the
dependence of $A$ in \eqref{eq:Ainvariant} on the parameter $\ddd$.)
By making the mild assumption that $\ddd \le 3/4$, this expression
simplifies to
\begin{equation} \label{eq:CCDup}
\rhoCCD (\ddd,x^0) \le
1-\frac{\ddd}{n(n(1-\ddd)+\ddd)}.
\end{equation}
On the other hand, Sun and Ye show the following lower bound on
$\rhoCCD(\ddd,x^0)$ (obtained by substituting from \eqref{eq:consts}
into Theorem~3.1 of \citet{SunY16a}):
\begin{equation} \label{eq:CCDlo}
\rhoCCD (\ddd,x^0) \geq \left( 1- \frac{2 \ddd
  \pi^2}{n(n(1-\ddd)+\ddd)} \right)^2.
\end{equation}
By combining \eqref{eq:CCDup} and \eqref{eq:CCDlo}, we see that for
small values of $\ddd/n$, the average epoch-wise decrease in error is
$\rhoCCD (\ddd,x^0) = 1- c \ddd/n^2$, for some moderate value of
$c$. Classical numerical analysis for Gauss-Seidel derives similar
dependency on $n^2$ for this case from the eigenvalues of $A$, $D$,
and $L$; see \citep{SamN89a},
\citet[p.~464]{YouR71a}, and
\citet[Theorem~3.44]{Hac16a}. This dependency on $n$ is confirmed
empirically, by running CCD for $A$ with the same $\ddd$ but different
$n$, as shown in Figure~\ref{fig:CCD_different_n}.

\begin{figure}\centering
	\begin{subfigure}[b]{0.32\textwidth}
\includegraphics[width=\linewidth]{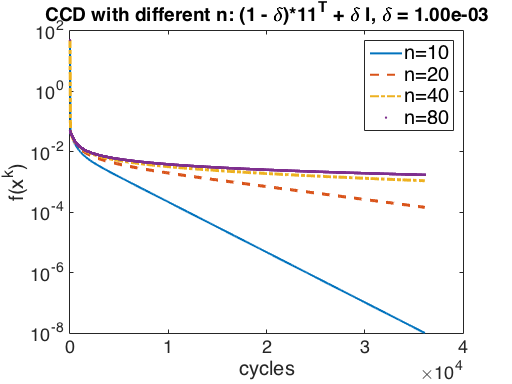}
\caption{CCD}
\label{fig:CCD_different_n}
\end{subfigure}
\begin{subfigure}[b]{0.32\textwidth}
\includegraphics[width=\linewidth]{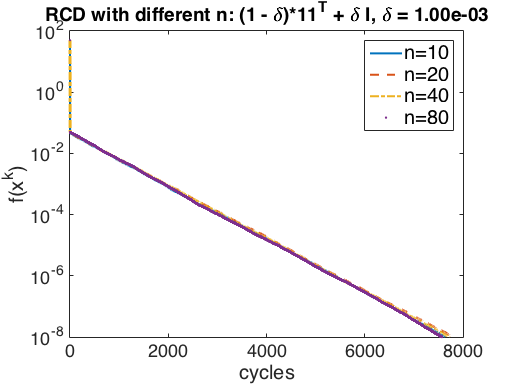}
\caption{RCD}
\label{fig:RCD_different_n}
\end{subfigure}
\begin{subfigure}[b]{0.32\textwidth}
\includegraphics[width=\linewidth]{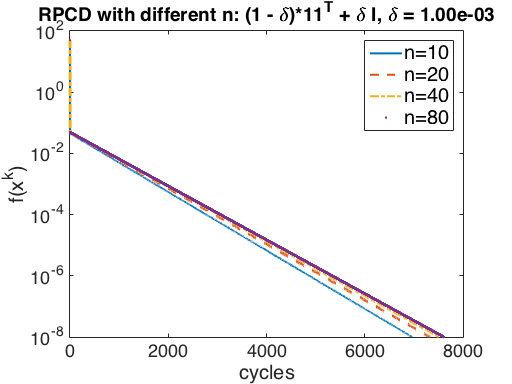}
\caption{RPCD}
\label{fig:RPCD_different_n}
\end{subfigure}
\caption{Convergence of $f$ for CCD, RCD, and RPCD applied to
  \eqref{eq:q}, \eqref{eq:Ainvariant}, with $\ddd = .001$ and $n=
  10,20,40,80$. Convergence rate of CCD deteriorates as $n$ grows, as
  predicted, while the convergence rates of RCD and RPCD are
  independent of $n$.\label{fig:different_n}}
\end{figure}

For RCD, we have by substituting the values in \eqref{eq:consts} into
\eqref{eq:rcd.rate} that the expected per-epoch improvement in error
is given by
\begin{equation}
	\label{eq:RCDrate}
	\rhoRCD(\ddd,\text{predicted}) \le \left(1 - \frac{\mu}{n
          \Lmax} \right)^n = \left(1-\frac{\ddd}{n} \right)^n
       \approx 1 - \ddd + O(\ddd^2).
\end{equation}
This result suggests that complexity of RCD is $O(n^2)$ times better
than CCD for small $\ddd$, and that its rate does not depend strongly
on $n$. This independence of $n$ is confirmed empirically by
Figure~\ref{fig:RCD_different_n}. The expression \eqref{eq:rcd.rate.R}
suggests a slightly better complexity for RCD of roughly $|\log
\epsilon| (1+\ddd) /(2 \ddd)$ epochs, rather than $| \log \epsilon|
/\ddd$ epochs, corresponding to replacing $1-\ddd$ in
\eqref{eq:RCDrate} by
\begin{equation} \label{eq:RCDrate.R}
\left(1-\frac{2 \ddd}{n(1+\ddd)} \right)^n \approx 1-\frac{2\ddd}{1+\ddd}.
\end{equation}

A kind of lower bound on the per-iterate improvement of RCD on the
problem \eqref{eq:q}, \eqref{eq:Ainvariant} can be found by setting
$x_i = (-1)^i$, $i=1,2,\dotsc,n$ with $n$ even. It can be shown that the function
values for this $x$ and the next RCD iterate $x^+$ are
\[
f(x) = \frac12 \ddd n, \quad f(x^+) = \frac12 \ddd (n-\ddd) = \left(1-\frac{\ddd}{n} \right) f(x), 
\]
regardless of the component of $x$ chosen for updating.  This
expression reveals a one-iteration improvement that matches the upper
bound \eqref{eq:rcd.rate}. However, as with some other lower-bound
examples, the longer-term behavior of the iteration sequence is more
difficult to predict. This same example provides a Q-linear rate in
the quantity $f(x) - f(x^*) + \sum_i L_i (x_i-x^*_i)^2/2$ of $(1 - 2 \ddd / (n(1 +
\ddd))$, exactly matching the upper bound of \eqref{eq:rcd.rate.R},
\eqref{eq:RCDrate.R}, proving that the R-linear rate of
\eqref{eq:rcd.rate.R} is also tight, in a sense.  A lower bound on
$\|\E (x^k)\|$ is proved in \citep{YosS15a}, but this does not
translate into a lower bound of the expected function value.

Figure~\ref{fig:RPCD_different_n} shows that RPCD too has a
convergence rate independent of $n$ on this matrix. (The performances
of RPCD and RCD are quite similar on the problems graphed.) The
convergence rate of CCD deteriorates with $n$, according to the
predictions above.

\section{Convergence of  RPCD for the Permutation-Invariant $A$} 
\label{sec:rpcd}

We now analyze the expected convergence behavior of RPCD on the convex
quadratic problem \eqref{eq:q} with permutation-invariant Hessian $A$
defined by \eqref{eq:Ainvariant}. We start by deriving a two-parameter
recurrence that captures the behavior of the method over each epoch,
and derive an estimate for the expected convergence of $f(x^{\ell n})$
to zero, as a function of these parameters. In our main results, we
analyze the rate of convergence of this sequence of parameter pairs to
zero.

\subsection{A Two-Parameter Recurrence} \label{sec:2param}

Since $A$ in \eqref{eq:Ainvariant} is invariant under symmetric
permutations, the matrices $L$ and $D$ are the same for all $P^T AP$,
where $P$ is any permutation matrix.  When considering
RPCD applied to this problem, we have in the notation of
\eqref{eq:Cl} that $C_{\ell} \equiv C$ for all $\ell$.  The
expression \eqref{eq:xRGS} simplifies as follows:
\begin{equation} \label{eq:xRGS.invariant}
\xRPCD^{\ell n} = P_{\ell} C P_{\ell}^T P_{\ell-1}
C P_{\ell-1}^T \dotsc P_1 C P_1^T x^0.
\end{equation}
The function values are
\begin{equation} \label{eq:fRGS.invariant}
f\left(\xRPCD^{\ell n}\right) = \frac12 \left(x^0\right)^T \left( P_1 C^T P_1^T
\dotsc P_{\ell} C^T P_{\ell}^T A P_{\ell} C P_{\ell}^T \dotsc P_1 C
P_1^T \right) x^0.
\end{equation}

We now analyze the expected value of the function
\eqref{eq:fRGS.invariant} obtained after $\ell$ epochs of RPCD, where
$A$ has the form \eqref{eq:Ainvariant}. Expectation is taken over the
independent permutation matrices $P_{\ell}, P_{\ell-1}, \dotsc, P_1$
in succession, followed finally by expectation over $x^0$. We define
$\bar{A}^{(t)}$, $t=0,1,2,\dotsc, \ell$ as follows:
\[
\bar{A}^{(t)} 
:= \E_{P_{\ell-t+1},\dotsc,P_{\ell}} \left(
P_{\ell-t+1} C^T P_{\ell-t+1}^T \dotsc P_{\ell} C^T P_{\ell}^T A
P_{\ell} C P_{\ell}^T \dotsc P_{\ell-t+1} C P_{\ell-t+1}^T \right),
\]
and note that $\bar{A}^{(0)} = A$ and (by comparison with
\eqref{eq:fRGS.invariant}) that 
\begin{equation} \label{eq:fRGS.x0}
\E  f\left(\xRPCD^{\ell n}\right) = \frac12 \E_{x^0} \left[
	\left(x^0\right)^T \bar{A}^{(\ell)} x^0 \right].
\end{equation}
We have the following recursive relationship between successive terms
in the sequence of $\bar{A}^{(t)}$ matrices:
\begin{equation} 
\label{eq:recursive}
\bar{A}^{(t)} = \E_{P_{\ell-t+1}} ( P_{\ell-t+1} C^T P_{\ell-t+1}^T \bar{A}^{(t-1)}
P_{\ell-t+1} C P_{\ell-t+1}^T )
= \E_{P} ( P C^T P^T \bar{A}^{(t-1)} P C
P^T ).
\end{equation}
(We can drop the subscript on $P_{\ell-t+1}$ since the permutation
matrices at each stage are i.i.d.)  We will show by a recursive
argument that each $\bar{A}^{(t)}$ has the form $\eta_t I + \nu_t
\bfone \bfone^T$, for some parameters $\eta_t$ and $\nu_t$.  Note that
for $\bar{A}^{(t)}$ of this form, we have that $P^T \bar{A}^{(t)} P =
\bar{A}^{(t)}$, a property that is crucial to our analysis.  We derive
a stationary iteration between the successive pairs
$(\eta_{t-1},\nu_{t-1})$ and $(\eta_t,\nu_t)$, and reveal the
convergence properties of RPCD by analyzing the $2 \times 2$ matrix
that relates successive pairs.

We start with a technical lemma.
\begin{lemma}
\label{lemma:form}
Given any matrix $Q \in \mathbb{R}^{n\times n}$ and permutation matrix
$P$ selected uniformly at random from the set of all permutations
$\Pi$, we have
\begin{equation}
\label{eq:form}
B:= \E_P[P Q P^T] = \tau_1 I + \tau_2 \bfone\bfone^T,
\end{equation}
where
\begin{equation}
\tau_2 = \frac{\bfone^T Q \bfone - \trace(Q)}{n(n-1)},\quad
\tau_1 = \frac{\trace (Q)}{n} - \tau_2.
\label{eq:alpha}
\end{equation}
\end{lemma}
\begin{proof}
For any $P\in \Pi$, if $P$ shifts the $i$th position to the $j$th
position, then $(PQP^T)_{jj} = Q_{ii}$. Since the probability of
taking any permutation from $\Pi$ is identical, we have that
\begin{equation*}
\cP((PQP^T)_{jj} = Q_{ii}) = \frac{1}{n}, \quad \forall i,j \in
\{1,\dotsc, n\}
\end{equation*}
(where $\cP(\cdot)$ denotes probability). Therefore, each diagonal
entry $B$ is the average over all diagonal entries of $Q$.
\begin{equation*}
B_{jj} = \frac{\sum_{i=1}^n Q_{ii}}{n}, \quad j= 1,2,\dotsc, n.
\end{equation*}
Consider permutations that shift the $i$th and the $j$th entries to
the $k$th and the $l$th positions, respectively, that is,
\begin{equation} \label{eq:p22}
(PQP^T)_{kl} = Q_{ij}.
\end{equation}
Note that we always have that $i\neq j \Rightarrow k\neq l$ because
permutations are bijections from and to $\{1,\dotsc,n\}$. Thus, there
are $(n-2)!$ permutations in $\Pi$ with the property \eqref{eq:p22}.  Under the
same reasoning as before, each off-diagonal entry of $B$ is the
average of all off-diagonal entries of $Q$.
\begin{equation*}
B_{kl} = \frac{\sum_{1\leq i,j \leq n, i\neq j} Q_{ij}}{n(n-1)},
\quad k,l \in \{1,2,\dotsc, n\}, \; k \neq l.
\end{equation*}
Finally, we obtain \eqref{eq:alpha} by noting that $B_{ii} = \tau_1 +
\tau_2$, while $B_{ij} = \tau_2$ for $i \neq j$.
\end{proof}


We have immediately from Lemma~\ref{lemma:form} that
\begin{equation} \label{eq:DF}
\E_P (P^TC^TCP) = d_1 I + d_2 \bfone \bfone^T, \quad
\E_P (P^T C^T \bfone \bfone^T CP) = m_1 I + m_2 \bfone \bfone^T,
\end{equation}
where
\begin{subequations} \label{eq:df}
\begin{alignat}{2}
d_2 &=
\frac{\bfone^T C^T C \bfone - \mbox{trace}(C^T C)}{n(n-1)} &&= 
\frac{\|C \bfone\|_2^2 - \| C \|_F^2 }{n(n-1)}
\label{eq:d2}\\
d_1
&= \frac{\mbox{trace}(C^T C)}{n} - d_2 &&=
\frac{\| C\|_F^2}{n} - d_2
\label{eq:d1}\\
m_2 &=\frac{(\bfone^T C \bfone)^2 - (\bfone^T C)(C^T\bfone)}{n(n-1)} &&=
\frac{(\bfone^T C \bfone)^2 - \| C^T \bfone \|_2^2 }{n(n-1)}
\label{eq:f2}\\
m_1
&= \frac{(\bfone^TC)(C^T\bfone)}{n} - m_2 &&=
\frac{\|C^T\bfone\|_2^2}{n-1} - \frac{(\bfone^T C \bfone)^2}{n(n-1)}.
\label{eq:f1}
\end{alignat}
\end{subequations}
Note that for \eqref{eq:f2} and \eqref{eq:f1} we used the property
$\text{trace}(AB) = \text{trace}(BA)$.

The following theorem reveals the relationship between successive
matrices in the sequence $\bar{A}^{(0)}, \bar{A}^{(1)}, \dotsc$.
\begin{theorem} \label{th:expectedobj}
Consider solving \eqref{eq:q} with the matrix $A$ defined in
\eqref{eq:Ainvariant} using RPCD.  For $\bar{A}^{(t)}$ defined in
\eqref{eq:recursive}, with $\bar{A}^{(0)}=A$, we have
\begin{equation}
\label{eq:Abar}
\bar{A}^{(t)} = \eta_{t} I + \nu_{t} \bfone \bfone^T,
\end{equation}
where $(\eta_0,\nu_0) = (\ddd,1-\ddd)$ and 
\begin{equation}
\label{eq:system}
\left[\begin{matrix} \eta_{t+1} \\ \nu_{t+1}\end{matrix}\right]
= M
\left[\begin{matrix} \eta_{t}\\
\nu_{t}\end{matrix}\right]
= M^{t+1}
\left[\begin{matrix} \ddd \\
1-\ddd \end{matrix}\right], \quad \forall t \ge 0,
\end{equation}
where
\begin{equation} \label{eq:def.M}
M := \left[ \begin{matrix} d_1 & m_1 \\ d_2 & m_2 \end{matrix} \right],
\end{equation}
and  $d_1, d_2, m_1, m_2$ are defined in \eqref{eq:df}.
\end{theorem}
\begin{proof}
We first prove \eqref{eq:Abar} by induction.  By
\eqref{eq:Ainvariant}, it holds at $t = 0$.  Now assume it holds for
$t = k$, for some $\eta_k$ and $\nu_k$, then for $k+1$ we have from
\eqref{eq:recursive}
\begin{equation*}
\bar{A}^{(k+1)} = \E_{P}[P C^T P^T \bar{A}^{(k)} P C P^T].
\end{equation*}
Because $\bar{A}^{(k)}$ is in the form \eqref{eq:Abar}, it is
invariant to row and column permutations, that is, $P^T\bar{A}^{(k)}P
= \bar{A}^{(k)}$ for all $P \in \Pi$.  Hence,
\begin{align}
\nonumber
\bar{A}^{(k+1)} & = \E_{P}[P C^T \bar{A}^{(k)} C P^T] \\
\nonumber
& = \eta_k \E_{P}[P C^T C P^T] + \nu_k \E_P[P C^T \bfone \bfone^T C
P^T] \\
\nonumber
& = \eta_k (d_1 I + d_2 \bfone \bfone^T)+ \nu_k (m_1 I + m_2 \bfone \bfone^T) \\
\label{eq:combination}
&= (\eta_k d_1 + \nu_k m_1)I + (\eta_k d_2 + \nu_k m_2) \bfone \bfone^T,
\end{align}
giving the result.
\end{proof}

We obtain a result for the expected value of $f$ after $\ell$ epochs
by taking the expectation as in \eqref{eq:fRGS.x0}, showing that the
sequence of expected function values at the end of each epoch is
governed by the behavior of the sequence $\{ (\eta_{\ell},\nu_{\ell})
\}$.
\begin{theorem} \label{th:RPCD.f}
Consider solving \eqref{eq:q} with the matrix $A$ defined in
\eqref{eq:Ainvariant} using RPCD.  Then, using the notation of
Theorem~\ref{th:expectedobj}, we have
\[
\E_{P_1,P_2,\dotsc,P_{\ell}} \, f(x^{\ell n}) = \frac12 \left( \eta_\ell \|x^0\|^2 +\nu_\ell (\bfone^Tx^0)^2 \right) \le   \left\|
\left[ \begin{matrix} \eta_{\ell} \\ \nu_{\ell}  \end{matrix} \right] \right\| \max \left( \|x^0\|^2, (\bfone^Tx^0)^2 \right).
\]
\end{theorem}
\begin{proof}
  The result is obtained by taking expectations with respect to
  $P_{\ell}, P_{\ell-1}, \dotsc, P_1$ in \eqref{eq:fRGS.invariant},
  and using the definition of $\bar{A}^{(t)}$ (with $t=\ell$) together
  with
  Theorem~\ref{th:expectedobj}.
\end{proof}

Figure~\ref{fig:expected} plots the expected value from
Theorem~\ref{th:RPCD.f} against the value of $f(\xRPCD^{\ell n})$
obtained from \eqref{eq:fRGS.invariant} for particular random choices
of $x^0$ and the permutation matrices $P_1,P_2, \dotsc, P_{\ell}$,
showing that the estimate in this one instance tracks the expected
value closely. (This behavior is typical.)


\begin{figure}
	\centering
	\includegraphics[width=0.7\linewidth]{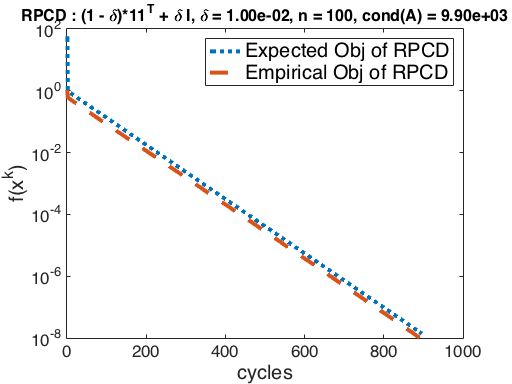}
	\caption{Empirical objective value and expected objective value of
	RPCD.}
	\label{fig:expected}
\end{figure}

\subsection{Convergence of the Two-Parameter Recurrence} \label{sec:2conv}

It is evident from Theorems~\ref{th:expectedobj} and \ref{th:RPCD.f}
(and Gelfand's formula) that the asymptotic convergence of the
expected value of $f$ is governed by $\rho(M)$, which, because of
definitions \eqref{eq:def.M} and \eqref{eq:df}, is a function of
$\ddd$ and $n$.  In Figure~\ref{subfig:deltalarge} of
Section~\ref{sec:pinvA} and Table~\ref{tab:table1} of
Section~\ref{sec:computations}, we see that this rate is significantly
better than those obtained for RCD and CCD when $\delta$ is not too
close to zero (that is, $\delta \ge .2$). In this section, we estimate
the convergence rate of RPCD for $\delta$ close to zero, showing that
in this regime, it is close to the rate of approximately $1- 2\ddd$
obtained by RCD \eqref{eq:RCDrate.R}, and much faster than the rate of
CCD discussed in \eqref{eq:CCDup}, \eqref{eq:CCDlo}, which is $1-c
\ddd/n^2$, for some modest value of $c$.

We start with a rigorous bound on the convergence rate for the
sequence $\{ (\eta_{\ell}, \nu_{\ell} ) \}$, without resorting to
eigenvalue calculations for $M$.  The details of bounding the elements
of $M$ ($d_1$, $d_2$, $m_1$, and $m_2$ from \eqref{eq:df}) as
functions of $\ddd$ and $n$ are shown in
Appendix~\ref{app:computation}. For the case of $\ddd \in [0,0.4]$ and
$n \ge 10$, the formulas \eqref{eq:c.}  yield the following bounds:
\begin{align*}
	0 \le d_1 &\le 1-2 \ddd +3.6 \ddd^2,
	\\
|m_2| & \le  .05 \ddd^2,\\
0 \le d_2 &\le 1 - 2\ddd + 3.2 \ddd^2,\\
|m_1| & \le  .15 \ddd^2.
\end{align*}
By appealing to Theorems~\ref{th:expectedobj} and \ref{th:RPCD.f}, we
obtain our main result.
\begin{theorem} \label{th:rpcd.main}
Consider solving \eqref{eq:q}, \eqref{eq:Ainvariant} with $\ddd \in
[0, 0.4]$ and $n\ge 10$ using RPCD. Then, using the notation of
Theorem~\ref{th:expectedobj}, we have that
\begin{equation} \label{eq:rec}
|\eta_{\ell}| \le (1-2\ddd+4\ddd^2)^{\ell-1} \delta, \quad
|\nu_{\ell}| \le (1-2\ddd+4\ddd^2)^{\ell-1} \delta, \quad \forall \ell \ge 1.
\end{equation}
Thus, we have the following bound on the convergence of the expected
value of the function:
\[
\E_{P_1,P_2,\dotsc,P_{\ell}} \, f(x^{\ell n}) \le \frac12 (1-2\ddd+4\ddd^2)^{\ell-1} \left(\|x^0\|^2 + (\bfone^Tx^0)^2 \right) \ddd, \quad \forall \ell \ge 1.
\]
\end{theorem}
\begin{proof}
Since $(\eta_0,\nu_0) = (\ddd,1-\ddd)$, we have from \eqref{eq:system}
and using $\ddd \in [0,0.4]$ that
\[
\left[ \begin{matrix} |\eta_1| \\  |\nu_1| \end{matrix} \right] \le
\left[ \begin{matrix} d_1 & |m_1| \\ d_2 & |m_2| \end{matrix} \right] 
\left[ \begin{matrix} \ddd  \\ (1-\ddd) \end{matrix} \right] \le
\left[ \begin{matrix} (1-2\ddd+3.6 \ddd^2) \ddd + .15 \ddd^2 (1-\ddd) \\
(1-2\ddd+3.2 \ddd^2) \ddd +  .05\ddd^2 (1-\ddd) \end{matrix} \right] 
\le
\left[ \begin{matrix} (1-1.85\ddd+3.6 \ddd^2) \ddd \\
(1-1.95\ddd+3.2 \ddd^2) \ddd \end{matrix} \right] 
\le 
\left[ \begin{matrix} \ddd \\ \ddd \end{matrix} \right],
\]
so that \eqref{eq:rec} holds for $\ell=1$. Supposing that the bound
holds for some value of $\ell \ge 1$, we have
\begin{align*}
\left[ \begin{matrix} |\eta_{\ell+1}| \\  |\nu_{\ell+1}| \end{matrix} \right] & \le
(1-2\ddd+4\ddd^2)^{\ell-1}
\left[ \begin{matrix} (1-2\ddd+3.6 \ddd^2) & .15 \ddd^2 \\
(1-2\ddd+3.2 \ddd^2) &   .05\ddd^2  \end{matrix} \right] 
\left[ \begin{matrix} \ddd \\ \ddd \end{matrix} \right] \\
& \le
(1-2\ddd+4\ddd^2)^{\ell -1}
\left[ \begin{matrix} 
(1-2\ddd+3.75 \ddd^2) \ddd \\
(1-2\ddd+3.25 \ddd^2) \ddd
 \end{matrix} \right] \\
& \le 
(1-2\ddd+4\ddd^2)^{\ell}
\left[ \begin{matrix} 
\ddd \\ \ddd
 \end{matrix} \right],
\end{align*}
verifying that the required bound still holds at $\ell+1$, thus proving
\eqref{eq:rec}.

The final claim follows directly from Theorem~\ref{th:RPCD.f}.
\end{proof}

This result indicates a global linear rate of at worst $1-2\ddd
+4\ddd^2$, similar to the rate \eqref{eq:RCDrate.R} obtained for RCD
(identical to $O(\ddd)$) and much faster than the rate obtained for
CCD in \eqref{eq:CCDup}, \eqref{eq:CCDlo}.

By using slightly more refined estimates of the elements of $M$,
which involve not strict upper bounds as in \eqref{eq:c.} but rather
remainder terms containing higher powers of $\ddd$ and/or $1/n$, we
can obtain an estimate of $\rho(M)$. In Appendix~\ref{app:m2}, we
obtain the following estimates of $d_1$, $d_2$, $m_1$, and $m_2$:
\begin{align*}
	d_1 &= 1-2 \ddd - 2 \frac{\ddd}{n} + 2 \ddd^2 +
	O \left(\frac{\ddd^2}{n} \right) +  O(\ddd^3),
	\\
m_2 &= O\left(\frac{\ddd^2}{n}\right),\\
d_2 &= 1 - \frac{2}{n} + O(\ddd),\\
m_1 &= O\left(\frac{\ddd^2}{n}\right).
\end{align*}
By substituting these estimates into \eqref{eq:def.M} and calculating
the spectral radius $\rho(M)$ as the largest root of the
characteristic quadratic $\det (M-\lambda I)$, we obtain
\begin{equation} \label{eq:rhoM}
	\rho(M) = 1-2\ddd-\frac{2\ddd}{n} + 2 \ddd^2 +
	O \left(\frac{\ddd^2}{n} \right) + O(\ddd^3).
\end{equation}
This asymptotic rate is identical to the rate for RCD
\eqref{eq:RCDrate.R} in the $1$, $\ddd$, and $\ddd^2$, terms, and is
slightly better because of the presence of the $-2\ddd/n$ term. 

\subsection{The First Iteration} \label{sec:first}

We noted in the numerical experiments (Figures~\ref{fig:different_n}
and \ref{fig:expected}) that the decrease in $f$ over the first epoch
of RPCD is rather dramatic. In fact, after just a single {\em
  iteration}, the function value was often of order $\ddd$, for all
three variants (CCD, RPCD, and RCD). The following result supports
this observation.
\begin{theorem} \label{th:first}
Consider solving \eqref{eq:q} with the matrix $A$ defined in
\eqref{eq:Ainvariant} using RCD or RPCD with exact line search
\eqref{eq:q.exact}. Given any $x^0$, the expected function value
after a single iteration satisfies
\begin{equation}
\label{eq:first}
\E_i f(x^1)
=\frac12 \ddd \left( 1- \frac{\ddd}{n} \right) \|x^0\|^2 +
\frac12 \ddd (1-\ddd)
\left( 1-\frac{2}{n} \right) (\bfone^Tx^0)^2  \le \frac12 \ddd \|x^0\|^2 + \ddd f(x^0),
\end{equation}
where $i$ denotes the coordinate chosen for updating at the first
iteration. 
\end{theorem}
\begin{proof}
Note that $i$ is chosen uniformly at random from $\{ 1,2,\dotsc,n\}$
for both RPCD and RCD. After one step of CD, we have
\begin{align*}
x^1_i & = x^0_i - \left(x^0_i + (1-\ddd) \sum_{j \ne i} x^0_j \right) = -(1-\ddd)
\left( \sum_{j \ne i} x^0_j \right); \\
 x^1_j & = x^0_j \;\; \mbox{for $j \ne i$.}
\end{align*}
Thus from \eqref{eq:Ainvariant} we have
\begin{align}
\nonumber
f\left(x^1\right) &= \frac12 \left(x^1\right)^TA x^1 \\
\nonumber
& =
 \frac12 \ddd \left\|x^1 \right\|^2 + \frac12 \left(1-\ddd\right) \left( \sum_{j=1}^n  x^1_j \right)^2 \\
\nonumber
&= \frac12 \ddd \left[ \sum_{j \neq i} \left(x^0_j\right)^2 +
	\left(1-\ddd\right)^2 
\left( \sum_{j \ne i} x^0_j \right)^2 \right]
 + \frac12 \left(1-\ddd\right)
\left[ \sum_{j \ne i} x^0_j - (1-\ddd) \sum_{j \ne i} x^0_j
  \right]^2  \\
\nonumber
&=\frac12 \ddd \sum_{j \neq i} \left(x^0_j\right)^2 + \left( \sum_{j \ne i} x^0_j \right)^2
\left[ \frac12 \ddd (1-\ddd)^2 + \frac12 \ddd^2 (1-\ddd)
\right] \\
\label{eq:first.iter.3}
&=\frac12 \ddd \sum_{j \neq i} \left(x^0_j\right)^2 + \frac12 \ddd (1-\ddd) \left( \sum_{j \ne i} x^0_j \right)^2.
\end{align}
Since 
\begin{align*}
\E_i \sum_{j \ne i} \left(x^0_j\right)^2  & = 
\frac{n-1}{n} \sum_{j=1}^n  \left(x^0_j\right)^2 =
\frac{n-1}{n} \|x^0 \|^2, \\ 
\E_i \left( \sum_{j \ne i} x_j^0 \right)^2 & = \E_i (\bfone^Tx^0-x_i)^2 \\
&=\E_i \left( (\bfone^Tx^0)^2 - 2x_i (\bfone^T x^0) + x_i^2 \right) \\
&= (\bfone^Tx)^2 - \frac{2}{n} (\bfone^Tx^0)^2 + \frac{1}{n} \|x^0\|^2 \\
&= \left(1-\frac{2}{n} \right) (\bfone^Tx^0)^2 + \frac{1}{n} \|x^0\|^2.
\end{align*}
we have by taking expectation with respect to $i$ in
\eqref{eq:first.iter.3} that the equality in \eqref{eq:first} holds.

For the inequality in \eqref{eq:first}, we have from
\[
f(x^0) = \frac12 (x^0)^T A x^0 = \frac12 \ddd \|x^0 \|^2 + \frac12 (1-\ddd) (\bfone^Tx^0)^2 \ge \frac12 (1-\ddd) (\bfone^Tx^0)^2
\]
that
\begin{align*}
\E_i f(x^1) &= \frac12 \ddd \left( 1- \frac{\ddd}{n} \right) \|x^0\|^2
+ \frac12 \ddd (1-\ddd) \left( 1-\frac{2}{n} \right) (\bfone^Tx^0)^2 \\
& \le \frac12 \ddd \| x^0\|^2 + \frac12 \ddd (1-\ddd) (\bfone^T x^0)^2 \\
& \le \frac12 \ddd \| x^0\|^2 + \ddd f(x^0),
\end{align*}
as required.
\end{proof}

For CCD, we have from \eqref{eq:first.iter.3} with $i=1$ that
\begin{align*}
  f(x^1) &= \frac12 \ddd ( \|x^0\|^2 - (x_1^0)^2) + \frac12 \ddd (1-\ddd) \left((\bfone^Tx^0) - x_1^0 \right)^2 \\
  & \le \frac12 \ddd \|x^0\|^2 + \frac12 \ddd (1-\ddd) \left[ (\bfone^Tx^0)^2 - 2(x_1^0) (\bfone^Tx^0) \right] \\
  & \le \frac12 \ddd \left[ \|x^0\|^2 + (\bfone^Tx^0)^2 + 2 \|x^0\| (\bfone^Tx^0) \right] \\
  &= \frac12 \ddd \left[ \|x^0\| + (\bfone^T x^0) \right]^2.
\end{align*}
If $x^0$ is independent of $\ddd$, we have that $f(x^1) =
O(\ddd)$. However there is no guarantee that $f(x^1)$ is substantially
smaller than $f(x^0)$. If $x^0$ is chosen ``adversarially'' in such a
way that $| \bfone^Tx^0| \ll \|x^0\|$, we may find that $f(x^1)$ is
not much smaller than $f(x^0)$. For random choices of $x^0$, however,
we would expect a significant decrease on the first iteration, similar
to that observed for RPCD and RCD.

\section{Computational Results} \label{sec:computations}

Some comparisons between empirical rates and rates predicted from the
analysis are shown in Table~\ref{tab:table1}, for $n=100$ and
different values of $\delta$. For the empirical rates
$\rhoCCD(\ddd,x^0)$, $\rhoRCD(\ddd,x^0)$, and $\rhoRPCD(\ddd,x^0)$, we
used formulas like \eqref{eq:GS.av}, but we took the average decrease
factor {\em only over the last 10 epochs}, so as to capture the
asymptotic rates and discount the early iterations. We used the
termination criterion $f(x^{\ell n}) - f^* \le 10^{-8}$.  For the
theoretical predictions, we used $\rho(C)^2$ for CCD (as suggested by
\eqref{eq:fGS.av}), the formula $\rhoRCD(\ddd,{\text{predicted}}) =
(1- 2\ddd/(n(1 + \ddd)))^n$ for RCD (from \eqref{eq:RCDrate.R}), and
$\rho(M)$ for RPCD (from \eqref{eq:rhoM}).
We note from this table that the theoretical predictions for CCD and
RPCD are quite sharp, even for values of $\ddd$ that are not
particularly small. For RCD, the empirical results are slightly better
than predicted by the theory when $\ddd$ is large. RPCD has the best
practical and theoretical asymptotic convergence of the three
variants, with the advantage increasing as $\ddd$ increases.

\begin{table}
\begin{center}
	\begin{tabular}{@{}l@{}|llllll@{}}
	$\ddd$  & $0.80$  & $0.50$  & $0.33$  & $0.20$  & $0.10$  &
	$0.03$\\ \hline
	$\rhoCCD(\ddd,x^0)$  & $0.9340$  & $0.9924$  & $0.9971$  & $0.9988$
	& $0.9995$  & $0.9998$\\
	$\rho(C)^2$  & $0.9342$  & $0.9924$  & $0.9971$  & $0.9988$  &
	$0.9995$  & $0.9999$\\
	\hline
	$\rhoRCD(\ddd,x^0)$  & $0.3146$  & $0.4764$  & $0.5945$  & $0.7059$
	& $0.8287$  & $0.9428$\\
	$\rhoRCD(\ddd,{\text{predicted}})$ & $0.4095$  & $0.5123$  &
	$0.6081$  & $0.7161$  & $0.8336$  & $0.9434$\\
	\hline
	$\rhoRPCD(\ddd,x^0)$  & $0.1054$  & $0.3306$  & $0.4929$  &
	$0.6615$  & $0.8178$  & $0.9415$\\
	$\rho(M)$  & $0.1162$  & $0.3289$  & $0.4994$  & $0.6635$  &
	$0.8164$  & $0.9412$ \\
	\hline
\end{tabular}
\end{center}
\caption{Observed and predicted per-epoch convergence rates for CCD,
  RCD, and RPCD, for various values of $\delta$. ($n=100$ in all
  experiments.) \label{tab:table1}}
\end{table}

\section{Conclusions} \label{sec:conclusions}

Recent work has shown that the problem \eqref{eq:q} with Hessian
matrix \eqref{eq:Ainvariant} is a case that reveals significant
differences in performance between cyclic and randomized variants of
coordinate descent. Here, we provide an analysis of the performance of
random-permutations cyclic coordinate descent that sharply predicts
the practical convergence behavior of this approach, showing an
asymptotic convergence rate that at least matches (and is even
slightly better than) that obtained by a random
sampling-with-replacement scheme.


Empirically, it appears that convex quadratic instances that reveal
differences between CCD, RCD, and RPCD are quite limited in scope,
with \eqref{eq:Ainvariant} being the canonical instance and the one
whose analysis is most tractable.  In work subsequent to this paper
\citep{WriL16b}, we analyze the case of quadratic convex $f$ in which
the Hessian has the form $\ddd I + (1-\ddd) uu^T$, where $u \in \R^n$
is a vector whose elements have magnitude not too different from
$1$. By a diagonal transformation, this matrix has the form $\ddd I +
(1-\ddd) \bfone \bfone^T + \epsilon D$, where $D$ is diagonal with
elements in the range $[0,1]$ and $\epsilon \ge 0$. Our analysis in
\citep{WriL16b} builds on the approach in this paper, but is somewhat
more complex; the exact two-variable recurrence of
Theorem~\ref{th:expectedobj} becomes an approximate recurrence
involving more terms.

\section*{Acknowledgments}

We thank two referees and the Editor-in-Chief for their comments
on earlier drafts, which caused us to improve the presentation and
sharpen the results of the paper.

\bibliographystyle{IMANUM-BIB} \bibliography{cdrefs}

\appendix
\numberwithin{equation}{section}

\section{Estimating Terms in the Recurrence Matrix $M$}
\label{app:computation}

Here we first find upper and (in some cases) lower bounds for the
following quantities, for the matrix $A$ given in
\eqref{eq:Ainvariant} and the corresponding value of $C$ defined in
\eqref{eq:LD} and \eqref{eq:Cinvariant}:
\begin{equation} \label{eq:Cq}
(\bfone^T C \bfone)^2, \quad 
\|C\bfone\|^2, \quad \|
C^T \bfone\|^2, \quad \|C\|_F^2.
\end{equation}
We then use these quantities to obtain bounds on $d_1$, $d_2$, $m_1$,
and $m_2$ from \eqref{eq:df}.
We assume throughout that $n \geq 10$ and $\ddd \in [0,0.4]$.


For $\bfone^T C \bfone$, we have from \eqref{eq:LD} and
\eqref{eq:Cinvariant} that
\begin{equation*}
\bfone^T C \bfone = (1 - \ddd) (\bfone^T \bar{L}) (E^T \bfone)
= (1 - \ddd) u^T v,
\end{equation*}
where $u = \bar{L}^T \bfone $ and $v = E^T \bfone$ have the following
components:
\[
v_i = n - i, \quad i=1,2,\dotsc,n,
\]
(from \eqref{eq:LD}) and
\[
u_i = -1 + (1-\ddd)\sum_{t=0}^{n-i-1} \ddd^t = -1 + (1-\ddd) \frac{1
	- \ddd^{n - i}}{1-\ddd} = -\ddd^{n-i}, \quad i=1,\ldots,n,
\]
(from \eqref{eq:Lbar2}).
For  $\ddd \in [0,0.4]$,  we have
\begin{align*}
0 \ge \bfone^T C \bfone &= -(1-\ddd) \sum_{i=1}^n (n-i)\ddd^{n-i}\\
&= -(1-\ddd) \sum_{i=1}^{n-1} i \ddd^i \\
&= -(1-\ddd) \sum_{i=1}^{n-1} \sum_{j=i}^{n-1} \ddd^j \\
& \ge -(1-\ddd) \sum_{i=1}^{n-1}\frac{\ddd^i}{1-\ddd} \\
&= -\sum_{i=1}^{n-1} \ddd^i \\
& \ge -\frac{\ddd}{1-\ddd} \ge -2 \ddd.
\end{align*}
Therefore, we have
\begin{equation}
\label{eq:c1.0}
0 \le (\bfone^T C \bfone)^2  \le 4 \ddd^2.
\end{equation}


We next seek an upper bound for $\| C^T \bfone \|_2^2$. We have from
\eqref{eq:C.elts} that
\begin{align*}
( C^T \bfone)_j &= \sum_{i=1}^{j-1} C_{ij} + \sum_{i=j}^n C_{ij} \\
&= -(1-\ddd) \sum_{i=1}^{j-1} \ddd^{i-1} + 
(1-\ddd) \sum_{i=j}^n (\ddd^{i-j} - \ddd^{i-1}) \\
&=-(1-\ddd) \sum_{i=1}^n \ddd^{i-1} + (1-\ddd) \sum_{t=0}^{n-j} \ddd^t \\
&= -(1-\ddd^n) + (1-\ddd) \frac{1-\ddd^{n-j+1}}{1-\ddd} \\
&= \ddd^n - \ddd^{n-j+1}.
\end{align*}
It follows that
\begin{align}
\nonumber
\| C^T \bfone \|_2^2 &= \sum_{j=1}^n (\ddd^n-\ddd^{n-j+1})^2 \\
\nonumber
& \le \sum_{j=1}^n (\ddd^{n-j+1})^2 \\
\nonumber
&= \sum_{j=1}^n \ddd^{2j} \\
\nonumber 
& \le \frac{\ddd^2}{1-\ddd^2} \\
\label{eq:c1.2}
& \le 1.34 \ddd^2.
\end{align}

We now use \eqref{eq:C.elts} to compute bounds on the other quantities
in \eqref{eq:Cq}.  We have
\begin{align*}
(C \bfone)_i &= \sum_{j=1}^i C_{ij} + \sum_{j=i+1}^n C_{ij} \\
&= (1-\ddd) \sum_{j=1}^i (\ddd^{i-j}-\ddd^{i-1}) - (1-\ddd) \sum_{j=i+1}^n \ddd^{i-1} \\
&= (1-\ddd) \sum_{j=1}^i \ddd^{i-j} - (1-\ddd) n \ddd^{i-1} \\
&=(1-\ddd^i) - n(1-\ddd) \ddd^{i-1} \\
&=1-n\ddd^{i-1}+(n-1)\ddd^i.
\end{align*}
We thus obtain
\begin{align}
	\nonumber
\| C \bfone \|_2^2 &= \sum_{i=1}^n  \left[ 1-2n\ddd^{i-1} + 2(n-1) \ddd^i
+ n^2 \ddd^{2i-2} - 2n(n-1) \ddd^{2i-1} + (n-1)^2 \ddd^{2i} \right] \\
	\nonumber
&= n + \left[ -2n+2(n-1)\ddd \right]\sum_{i=1}^n \ddd^{i-1} 
+ \left[ n^2-2n(n-1)\ddd+(n-1)^2\ddd^2 \right] \sum_{i=1}^n (\ddd^2)^{i-1} \\
\label{eq:c1.10}
&=  n+ \left[ -2n+2(n-1)\ddd \right]  \frac{1-\ddd^n}{1-\ddd} +
\left[ n^2-2n(n-1)\ddd+(n-1)^2\ddd^2 \right] \frac{1-\ddd^{2n}}{1-\ddd^2}.
\end{align}
Noting that $[-2n+2(n-1)\ddd]<0$ and
$[n^2-2n(n-1)\ddd+(n-1)^2\ddd^2]>0$ for the values of $\ddd$ and $n$
of interest, and using $2n\ddd^n(1-\ddd^2) \le (2 \ddd^8)
n \ddd^2 \le .01 n \ddd^2$, we continue as follows:
\begin{align}
\nonumber
\| C \bfone \|_2^2 &\le n+ \left[ -2n+2(n-1)\ddd \right] \frac{1}{1-\ddd} 
+ \ddd^n \left[ 2n-2(n-1)\ddd \right] + \left[ n^2-2n(n-1)\ddd+(n-1)^2\ddd^2 \right]  \frac{1}{1-\ddd^2} \\
\nonumber
& \le \frac{1}{1-\ddd^2} \left[ n(1-\ddd^2) + \left[-2n+2(n-1) \ddd \right] (1+\ddd) + .01n \ddd^2 + n^2 -2n(n-1) \ddd + (n-1)^2 \ddd^2 \right] \\
\nonumber
& = \frac{1}{1-\ddd^2} \left[ n^2 -2n^2 \ddd + 2n \ddd - n - 2\ddd + \ddd^2 \left[-n+2(n-1)+.01n + (n-1)^2 \right]\right] \\
\nonumber
& \le \frac{1}{1-\ddd^2} \left[ n(n-1) (1-2\ddd) + n^2 \ddd^2  \right].
\end{align}
Thus, dividing by $n(n-1)$, and using $\ddd \in [0,0.4]$ to deduce
that $(1-\ddd^2)^{-1} \le 1+1.5 \ddd^2$, we obtain
\begin{align}
\nonumber
\frac{\| C \bfone \|_2^2}{n(n-1)} & \le
\frac{1}{1-\ddd^2} \left[(1-2\ddd) + \frac{n}{n-1} \ddd^2 \right] \\
\nonumber
& \le (1+1.5 \ddd^2) [(1-2\ddd) + 1.12 \ddd^2] \\
\nonumber
& \le (1-2\ddd) + (1.5+1.12) \ddd^2 + 2 \ddd^4 \\
\nonumber
& \le (1-2\ddd) + (1.5+1.12 + .5) \ddd^2  \\
\label{eq:c1.1}
&\le (1-2\ddd) + 3.2 \ddd^2.
\end{align}
For the corresponding lower bound, we pick up from \eqref{eq:c1.10}
and again use $[-2n+2(n-1)\ddd]<0$ and
$[n^2-2n(n-1)\ddd+(n-1)^2\ddd^2]>0$, together with
$[n^2-2n(n-1)\ddd+(n-1)^2\ddd^2] \le n^2(1+\ddd^2)$ and
$\ddd^{2n}(1+\ddd^2) \le .001 \ddd^2$, to obtain the following:
\begin{align} 
\nonumber
\| C \bfone \|_2^2 & \ge 
n + \left[ -2n+2(n-1)\ddd \right] \frac{1}{1-\ddd} + 
\left[ n^2 - 2n(n-1) \ddd + (n-1)^2 \ddd^2 \right] \frac{1}{1-\ddd^2} - \ddd^{2n} n^2 (1+\ddd^2)  \\
\nonumber
&= \frac{1}{1-\ddd^2} \left[ n(1-\ddd^2) + \left[ -2n+2(n-1)\ddd \right] (1+\ddd)+ 
\left[ n^2 -2n(n-1) \ddd + (n-1)^2 \ddd^2 \right]  \right] - .001 n^2 \ddd^2 \\
\nonumber
&= \frac{1}{1-\ddd^2} \left[ (n^2-n) + \left[ 2(n-1) -2n-2n(n-1) \right] \ddd
+ \left[ -n+2(n-1) +(n-1)^2 \right] \ddd^2 \right] - .001n^2 \ddd^2 \\
\nonumber
&= \frac{1}{1-\ddd^2} \left[ (n^2-n) + (-2n^2+2n-2) \ddd + (n^2-n-1) \ddd^2 \right] - .001 n^2 \ddd^2 \\
\nonumber
& \ge \frac{1}{1-\ddd^2} \left[ n(n-1) -2n(n-1) \ddd -2 \ddd + n(n-1) \ddd^2 - \ddd^2 -.001 n^2 \ddd^2 \right].
\end{align}
Thus, dividing by $n(n-1)$, we obtain
\begin{align}
\nonumber
\frac{\| C \bfone \|_2^2}{n(n-1)} & \ge \frac{1}{1-\ddd^2} \left[1-2\ddd + \ddd^2 - \frac{2\ddd+\ddd^2}{n(n-1)} - \frac{.001n^2}{n(n-1)} \ddd^2  \right] \\
\nonumber
& \ge \frac{1}{1-\ddd^2} \left[1-2\ddd + \ddd^2 - \frac{2.5 \ddd}{n(n-1)} - .002 \ddd^2  \right] \\
& \ge 1-2\ddd + .998 \ddd^2 - \frac{3 \ddd}{n^2}.
\label{eq:c1.1.lower}
\end{align}
Note that this lower bound is strictly positive in the regime $\ddd \in
[0,0.4]$ and $n \ge 10$.


For $\| C \|_F^2$, we obtain from \eqref{eq:C.elts} that
\begin{align}
	\nonumber
\frac{1}{(1-\ddd)^2} \|C\|_F^2
&=
\sum_{j=1}^n \left\{
\sum_{i=1}^{j-1} \ddd^{2i-2} +
\sum_{i=j}^n \left(\ddd^{2i-2j} - 2 \ddd^{2i-j-1} + \ddd^{2i-2}\right)    \right\} \\
	\nonumber
&= \sum_{j=1}^n \left\{
\frac{1-\ddd^{2j-2}}{1-\ddd^2} + \left(1-2 \ddd^{j-1} + \ddd^{2j-2}\right) 
\sum_{i=0}^{n-j} \ddd^{2i}
\right\} \\
\label{eq:Cftmp}
&= \sum_{j=1}^n \left\{
\frac{1-\ddd^{2j-2}}{1-\ddd^2} + \left(1-2 \ddd^{j-1} + \ddd^{2j-2}\right) 
\frac{1-\ddd^{2n-2j+2}}{1-\ddd^2}  
\right\} \\
	\nonumber
& \le \frac{1}{1-\ddd^2} \sum_{j=1}^n \left\{
\left(1-\ddd^{2j-2}\right) + \left(1- 2\ddd^{j-1} + \ddd^{2j-2}\right) 
\right\} \\
	\nonumber
& \le  \frac{1}{1-\ddd^2} \left\{ 2n-2 \frac{1-\ddd^n}{1-\ddd} \right\},
\end{align}
so that
\begin{align}
\nonumber
\| C \|_F^2 & \le \frac{1-\ddd}{1+\ddd}  \left\{ 2n-2 \frac{1-\ddd^n}{1-\ddd} \right\} \\ 
\nonumber
& \le \frac{1}{1+\ddd} \left\{ 2n-2n\ddd-2+2\ddd^n\right\} \\
\label{eq:Cftmp2}
& \le (1-\ddd+\ddd^2) \left\{ 2n-2n\ddd-2+.01\ddd^2\right\} \\
\nonumber
&= 2(n-1) + [-2(n-1)-2n] \ddd + [2(n-1)+2n+.01]\ddd^2 - [2n+.01] \ddd^3 + .01 \ddd^4 \\ 
\nonumber
& \le 2(n-1) + [-4n+2]\ddd + 4n \ddd^2,
\end{align}
where in \eqref{eq:Cftmp2} we used $2 \ddd^n \le  2 (\ddd^8) \ddd^2 \le .01 \ddd^2$.
It therefore follows that
\begin{equation} \label{eq:c1.3}
\frac{\|C\|_F^2}{n-1}  \le 2-4 \ddd- \frac{2\ddd}{n-1} + \frac{4n}{n-1} \ddd^2 
\le 2-4\ddd - \frac{2\ddd}{n} + 4.5 \ddd^2.
\end{equation}
It follows, using again $\ddd \in [0,0.4]$ and $n \ge 10$, that
\begin{align*}
\frac{\|C\|_F^2}{n-1}  & \le 2-4\ddd - \frac{2\ddd}{n} + 4.5 \ddd^2 \\
& \le 2-4\ddd - \frac{6\ddd}{n^2} + .998 n \ddd^2 \\
& \le 2 \left( 1-2\ddd-\frac{3\ddd}{n^2}\right) + .998 n \ddd^2 \\
& \le n \left( 1-2\ddd -\frac{3\ddd}{n^2} + .998 \ddd^2 \right)  \le \frac{\| C \bfone\|_2^2}{n-1},
\end{align*}
where we used \eqref{eq:c1.1.lower} for the final inequality. It
follows that
\begin{equation} \label{eq:c1.4}
\| C \bfone \|_2^2 - \|C \|_F^2 \ge 0.
\end{equation}

From the formulas \eqref{eq:df} together with \eqref{eq:c1.0},
\eqref{eq:c1.1}, \eqref{eq:c1.1.lower}, \eqref{eq:c1.2},
\eqref{eq:c1.3}, and \eqref{eq:c1.4}, and using $n \ge 10$, we have the
following: \begin{subequations} \label{eq:c.}
\begin{align}
\label{eq:c.d2}
0 \le d_2 &= \frac{\| C \bfone\|_2^2 - \| C \|_F^2}{n(n-1)}
\le \frac{\| C \bfone\|_2^2}{n(n-1)}
\le 1-2\ddd+3.2\ddd^2;  \\
\nonumber
0 \le d_1 &= \frac{\|C\|_F^2}{n-1} - \frac{\|C\bfone\|_2^2}{n(n-1)}\\
\nonumber
& \le \left( 2-4\ddd-\frac{2\ddd}{n} + 4.5 \ddd^2 \right) -
\left( 1-2\ddd + .998 \ddd^2 - \frac{3\ddd}{n^2} \right) \\
\nonumber
& \le 1-2 \ddd - \frac{2\ddd}{n} + \frac{3\ddd}{n^2} + 3.6 \ddd^2 \\
\nonumber
& \le 1-2\ddd  - \frac{\ddd}{n} (2-3/n) + 3.6 \ddd^2 \\
\label{eq:c.d1}
& \le 1-2\ddd + 3.6 \ddd^2; \\
\nonumber
|m_2| &= \left| \frac{(\bfone^TC\bfone)^2 - \| C^T \bfone\|_2^2}{n(n-1)} \right| \\
\nonumber
&\le   \frac{\max \left( (\bfone^TC \bfone)^2, \| C^T \bfone \|^2 \right)}{n(n-1)} \\
\label{eq:c.f2}
&\le \frac{4 \ddd^2}{n(n-1)}  \le .05 \ddd^2; \\
\nonumber
|m_1| &=  \left| \frac{\|C^T\bfone\|_2^2}{n-1} - \frac{(\bfone^T C \bfone)^2}{n(n-1)} \right| \\
\label{eq:m1tmp}
& \le \max \left( \frac{\|C^T\bfone\|_2^2}{n-1} , \frac{(\bfone^T C \bfone)^2}{n(n-1)} \right) \\
\label{eq:c.f1}
& \le \max \left( \frac{1.34}{9}, \frac{4}{90} \right) \ddd^2 \le .15 \ddd^2.
\end{align}
\end{subequations}


\section{Approximation of $d_1, d_2, m_1$, and $m_2$ for Estimating $\rho(M)$}
\label{app:m2}

From \eqref{eq:c.f2}, \eqref{eq:m1tmp}, \eqref{eq:c1.0} and
\eqref{eq:c1.2}, we have
\begin{equation}
	\label{eq:m12}
	m_1 = O\left( \frac{\ddd^2}{n} \right), \quad
	m_2 = O\left( \frac{\ddd^2}{n^2} \right).
\end{equation}

For the two terms $d_1$ and $d_2$, we first need better approximations
of $\|C\bfone\|_2^2$ and $\|C\|_F^2$.
From \eqref{eq:c1.10}, we proceed with
\begin{align}
	\nonumber
	\|C\bfone\|_2^2 &= n + (1 + \ddd + \ddd^2) \left[ -2 n + 2(n-1) \ddd \right]
	+ \left[  n^2 - 2 n (n-1)\ddd + (n-1)^2\ddd^2\right] (1 + \ddd^2)
	+ O\left(n^2 \ddd^3 \right)\\
	&= n(n-1) + \ddd \left( -2n^2 + 2n - 2 \right) + \ddd^2 \left( 2n^2
	- 2n - 1\right) + O\left( n^2 \ddd^3 \right).
	\label{eq:c1.11}
\end{align}
For $\|C\|_F^2$, we obtain from \eqref{eq:Cftmp} that
\begin{align*}
\frac{1}{(1-\ddd)^2} \|C\|_F^2
&= \frac{1}{1-\ddd^2} \sum_{j=1}^n \left\{
	2 - (\ddd^2)^{j-1} - 2 \ddd^{j-1} + \ddd^{2j-2} - \ddd^{2n-2j+2} +2\ddd^{2n-j+1} -\ddd^{2n} \right\} \\
&= \frac{1}{1-\ddd^2}
\left\{ 2n- (1+\ddd^2) -2 (1+\ddd+\ddd^2) + (1+\ddd^2) - \ddd^2 + O(\ddd^3) \right\}
\\
&= \frac{1}{1-\ddd^2} \left\{
2n-2-2\ddd-3\ddd^2  \right\} +O(\ddd^3),
\end{align*}
so that
\begin{align}
\nonumber
\| C \|_F^2 & = \frac{1-\ddd}{1+\ddd} (2n-2-2\ddd-3\ddd^2) +O(\ddd^3)\\
\nonumber
&= (1-\ddd)(1-\ddd+\ddd^2) (2n-2-2\ddd-3\ddd^2) + O(n\ddd^3) \\
\nonumber
&= (1-2\ddd+2\ddd^2)  (2n-2-2\ddd-3\ddd^2) + O(n\ddd^3) \\
\label{eq:cf}
&= (2n-2) - \ddd(4n-2) + \ddd^2 (4n-3) + O(n\ddd^3).
\end{align}
We then have from \eqref{eq:cf} and \eqref{eq:c1.11} that
\begin{subequations} \label{eq:c.new}
\begin{align}
\label{eq:c.new.d2}
d_2 &= \frac{\| C \bfone\|_2^2 - \| C \|_F^2}{n(n-1)}
= \frac{n(n-1) + O\left( n^2\ddd \right) - (2n-2)}{n(n-1)}
= 1-\frac{2}{n} + O\left( \ddd \right).\\
\nonumber
d_1
&= \frac{\|C\|_F^2}{n-1} - \frac{\|C\bfone\|_2^2}{n(n-1)}\\
\nonumber
&= 2 - 4 \ddd + 4 \ddd^2
+ \frac{-2\ddd}{n-1}
- \left(1 - 2\ddd + \frac{-2\ddd}{n(n-1)}
+ 2 \ddd^2 \right) + O \left(\frac{\ddd^2}{n} \right) +
O(\ddd^3)\\
&= 1  -2 \ddd - \frac{2\ddd}{n} + 2 \ddd^2 + O \left(\frac{\ddd^2}{n} \right)
+O(\ddd^3).
\label{eq:c.new.d1}
\end{align}
\end{subequations}

\section{Condition \eqref{eq:strongProperty} for $g(Ex)$ with $g$ Strongly Convex}
\label{app:gEx}

Suppose that $f(x)=g(Ex)$ where $g$ is strongly convex with modulus of
convexity $\sigma>0$, and $E \in \R^{m \times n}$. If $E=0$, all $x$
are optimal, so the claim \eqref{eq:strongProperty} holds
trivially. Otherwise, we have that $\sigma_{\text{min,nz}}$, the
minimum nonzero singular value of $E$, is strictly positive.

By strong convexity of $g$, there exists a unique $t^* \in \R^m$ such
that the solution set for \eqref{eq:f} has the form $\{ x \, | \,
Ex=t^* \}$. Let $P(x)$ denote the projection of any vector $x \in
\R^n$ onto this set. We have by Hoffman's Lemma \citep{Hof52} that
\[
\| x-P(x) \| \le \sigma_{\text{min,nz}}^{-1} \| E(x-P(x)) \| =
\sigma_{\text{min,nz}}^{-1} \| Ex-t^* \|.
\]
Thus by strong convexity, we have
\begin{equation} \label{eq:sc.1}
f(x) = g(Ex) \ge g(t^*) + \frac{\sigma}{2} \| E(x-P(x)) \|^2 \ge f^* + 
\frac{\sigma \sigma_{\text{min,nz}}^2}{2} \| x-P(x) \|^2.
\end{equation}
Meanwhile we have by convexity of $f$ that
\[
f^* \ge f(x) + \nabla f(x)^T(P(x)-x),
\]
so that
\[
f(x) - f^*   \le \| \nabla f(x) \| \| P(x)-x \| \le \| \nabla f(x) \| 
\left( \frac{2}{\sigma \sigma_{\text{min,nz}}^2} \right)^{1/2}(f(x)-f^*)^{1/2}.
\]
Dividing both sides by $(f(x)-f^*)^{1/2}$ we obtain
\[
\| \nabla f(x) \| \left( \frac{2}{\sigma \sigma_{\text{min,nz}}^2}
\right)^{1/2} \ge (f(x)-f^*)^{1/2} \;\; \Rightarrow
\| \nabla f(x) \|^2 \ge \frac{\sigma \sigma_{\text{min,nz}}^2}{2} (f(x)-f^*),
\]
which has the form \eqref{eq:strongProperty}.

\end{document}